\documentclass[11pt]{amsart}
\usepackage{latexsym,enumerate}
\usepackage{amsmath,amsthm,amsopn,amstext,amscd,amsfonts,amssymb}
\usepackage{color}
\numberwithin{equation}{section}

\newtheorem{theorem}{Theorem}

\newtheorem{lemma}{Lemma}
\newtheorem{corollary}{Corollary}
\newtheorem{proposition}{Proposition}

\newtheorem{remark}{Remark}

\newtheorem{definition}{Definition}

\numberwithin{theorem}{section}
\numberwithin{corollary}{section}
\numberwithin{lemma}{section}
\numberwithin{definition}{section}
\numberwithin{proposition}{section}
\numberwithin{remark}{section}

\newcommand{\R}{\mathbb R}
\newcommand{\N}{\mathbb N}

\newcommand{\medint}{-\kern  -,375cm\int}

\newcommand{\norm}[1]{\left\Vert#1\right\Vert}

\begin{document}
\title[ ]{Symmetry and asymmetry of minimizers \\ of a class of noncoercive functionals}
\author{F. Brock,  G. Croce, O. Guib\'e, A. Mercaldo }
\thanks{}
\date{}

\begin{abstract}
In this paper we prove symmetry results for  minimizers of a non coercive functional
defined on the class of Sobolev  functions  with zero mean value. We prove that the minimizers  are foliated
Schwarz symmetric, i.e. they are axially symmetric with respect to an axis passing through the
origin and nonincreasing in the polar angle from this axis. 
In the two dimensional case we show a symmetry breaking.
\end{abstract}

\maketitle

\setcounter{footnote}{1} \footnotetext{%
Leipzig University, Department of Mathematics, Augustusplatz, 04109 Leipzig,
Germany, e-mail: brock@math.uni-leipzig.de}

\setcounter{footnote}{2} \footnotetext{%
Normandie Univ, France; ULH, LMAH, F-76600 Le Havre; FR CNRS 3335, 25 rue Philippe Lebon 76600 Le Havre, France, e-mail: gisella.croce@univ-lehavre.fr}

\setcounter{footnote}{3} \footnotetext{%
  Laboratoire de Math\'ematiques Rapha\"el CNRS -- Universit\'e de Rouen, 
Avenue de l'Universit\'e, BP.12,
76801 Saint-\'Etienne du Rouvray, 
France, e-mail: olivier.guibe@univ-rouen.fr}

\setcounter{footnote}{4} \footnotetext{%
Dipartimento di Matematica e Applicazioni \textquotedblleft R.
Caccioppoli\textquotedblright , Universit\`{a} degli Studi di Napoli
``Fe\-derico II", Complesso Monte S. Angelo, via Cintia, 80126 Napoli,
Italy, e-mail: mercaldo@unina.it}

\section{Introduction}

Consider the functional
$$
v \in H^1_0(\Omega)\to \frac 12 \int_{\Omega}|\nabla v|^2
$$
subjected to the constraint $\displaystyle \int_{\Omega} v^2=1$, where $\Omega$ is the unit ball in the plane. Its critical values are the eigenvalues of 
 the classical fixed membrane  problem
\begin{equation}\label{membrane_problem}
\begin{cases}
-\Delta u=\lambda u,& \mbox{in}\, \,\Omega\,,
\\
u=0,&  \mbox{on}\,\,\partial\Omega\,.
\end{cases}
\end{equation}
It is known that  the first eigenfunctions are positive and Schwarz symmetric, 
that is, radial and decreasing in the radial variable.
On the contrary, the second eigenfunctions are sign-changing; they are not radial, but
they are  symmetric with respect to the reflection at some line $\R e$, and they are decreasing in the angle $\arccos[\frac{x}{|x|}\cdot e]\in (0,\pi)$. 
These
properties can be seen as a spherical version of the Schwarz symmetry along the foliation
of the underlying ball $\Omega$ by circles. For this reason, this  property has been called foliated Schwarz symmetry in the literature.

In the last years much interest has been devoted to the shape of sign changing minimizers of integral functionals, see for example \cite{Girao-Weth}, \cite{Weth}, \cite{BDNT}
and \cite{Parini-Weth}. 
In \cite{Girao-Weth}, Girao and Weth studied the symmetry properties of the minimizers of the problem
\begin{equation}\label{girao-weth}
v\to \frac{\norm{\nabla v}_{2}}{\norm{v}_{p}}, \quad v \in H^1(\Omega)\,,\,\,\,\,\,\,\,\int_{\Omega} v=0
\end{equation}
for $2\leq p<2^*$. In view of the zero average constraint, (\ref{girao-weth}) is similar to the problem of finding the second eigenfunctions of problem (\ref{membrane_problem}).
They proved that the minimizers are foliated Schwarz symmetric.

In \cite{Girao-Weth} Girao and Weth   pointed out another interesting phenomenon related to the shape of the minimizers of  (\ref{girao-weth}).
If $p$ is close to 2, then any minimizer of the above functional  is antisymmetric with respect to the reflection
at the hyperplane $\{x\cdot e=0\}$.
In contrast to this, the minimizers are not antisymmetric when $N=2$ and $p$ is
sufficiently large.
A similar break of symmetry was already observed  in 
\cite{DGS}, \cite{CD}, \cite{K},
\cite{BK}, \cite{BKN}, \cite{N} for the minimizers of
the functional
\[
v\to \frac{\norm{v'}_{L^p(0,1)}}{\norm{v}_{L^q(0,1)}}\,,\quad v \in W^{1,p}((0,1)), v(0)=v(1), \quad \int_0^1 v=0\,. 
\]
 Indeed, it has been shown that any minimizer is an antisymmetric function, if and only if $q\leq 3p$.

In this paper, we will prove similar symmetry results for the minimizers of a generalized version of the functional studied by Girao and Weth in \cite{Girao-Weth}. 
We consider
\begin{equation}\label{p0}
\lambda^{\theta,p}(\Omega)=\hbox{inf} 
\left\{
{\int_{\Omega}
\frac{|\nabla v|^2}{(1+|v|)^{2\theta}}} \,dx,\,\, 
v\in W^{1,q}(\Omega), \, v\neq 0 \\, \int_{\Omega}v\, dx=0, \, \norm{v}_{L^p(\Omega)}=1 
\right\}
\end{equation}
where $\Omega$ is either a ball or an annulus centered in the origin in $\R^N$, $N\geq 2$,  $\theta$ and $q$ satisfy 
\begin{gather}\label{teta}
0<2\theta<1\, ,
\\
\label{q3}
q=\frac{2N(1-\theta)}{N-2\theta}\,, \qquad \hbox{ if }  N\geq 3\,, 
\\
\label{q2}
2(1-\theta)\le q<2\, ,  \qquad   \hbox{ if }   N=2\,,  
\\
\label{qp}
1< p<q^* \qquad \text{ if } N\geq 3\,,
\\
\label{qpa}
1<p<+\infty \qquad\text{ if } N=2\,.
 \end{gather}
Observe that, if one defines
$$
 \Psi (\xi ) :=  
 \int_0 ^{\xi } (1+|t|)^{-\theta } dt = 
 \frac{\mbox{sgn}\, \xi }{1-\theta } [ (1+|\xi |)^{1-\theta } -1],\qquad \xi\in \R
$$
then our functional is the integral of 
$
 |\nabla \Psi(u)|^2 $, that is, \eqref{p0} is equivalent to 
 \begin{equation}\label{pnew}
\lambda^{\theta,p}(\Omega)=\hbox{inf} 
\left\{
{\int_{\Omega}
 |\nabla \Psi(v)|^2} \,dx,\,\, 
v\in W^{1,q}(\Omega), \, v\neq 0 \\, \int_{\Omega}v\, dx=0, \, \norm{v}_{L^p(\Omega)}=1 
\right\}
\end{equation}

The main feature of this functional is that it is not coercive on $H^1_0(\Omega)$, even if it is well defined on this Sobolev space. The lack of coercivity  has unpleasant consequences for the minimizers  
of  
\[
\displaystyle v\to \int_{\Omega}
\left[\frac{|\nabla v|^2}{(1+|v|)^{2\theta}}-G(x,v)\right] \,dx
\]
for functions $G$ having various growth assumptions.
Indeed, it was shown in 
\cite{BO},
\cite{ABO2001},
\cite{Mercaldo2003},
\cite{Faraci2004},
\cite{Porretta2007},
\cite{BCOLincei}, 
that the minimizers are  less regular than the minimizers of coercive functionals on $H^1(\Omega)$.

After recalling the definition of  foliated  Schwarz symmetry and proving some new sufficient conditions for this symmetry  in Section \ref{section_definition_symmetry}, 
we will prove the
foliated  Schwarz symmetry of the minimizers for $N\geq 2$.  As already pointed out, the same result has been obtained  by Girao and Weth in \cite{Girao-Weth} in the ``coercive" case, that is, 
for $\theta=0$.
We observe that in their proof, Girao and Weth make use of a well-known regularity result of the solutions of the Euler equation. In our case, we have to prove the analogous  regularity result for our  non coercive functional (see Section \ref{section_foliated_symmetry}).  Actually we are able to prove the
foliated  Schwarz symmetry of the minimizers of a more general functional, that is we consider 
\begin{equation}\label{p}
\lambda^{\theta,p}(\Omega)=\hbox{inf} 
\left\{
{\int_{\Omega}
\frac{|\nabla v|^2-F(|x|,v)}{(1+|v|)^{2\theta}}} \,dx,\,\, 
v\in W^{1,q}(\Omega), \, v\neq 0 \\, \int_{\Omega}v\, dx=0, \, \norm{v}_{L^p(\Omega)}=1 
\right\}
\end{equation}
where we assume that 
 $F:\mathbb{R}^{+} \times \mathbb{R}\rightarrow \mathbb{R}$ is a measurable
 function in $r=|x|\in [0,+\infty)$  and continuously differentiable in $t\in \mathbb{R}$,  which satisfies 
 \begin{equation}\label{F0}
F(r,0)=0\,,
\end{equation}
and the growth conditions  
\begin{gather}
  \label{growth}
|F(r,t)|\leq c(1+|t|)^p\,,\qquad c>0\,,
\\
\label{growthFt}
|F_t(r,t)|\le C_1(1+|t|)^{p-1} \,,\qquad C_1>0\,,
\end{gather}
for any $r\in [0,+\infty)$, $t\in \R$.

\noindent If $p\in (1,2)$, we add the requirement  
\begin{equation}
\label{condF} 
t(1+|t|) F_t (r,t ) - 2\theta |t| F(r,t ) \leq 0\,,
\end{equation}
for any $r\in [0,+\infty)$, $t\in \R$.

In the last two sections we will focus on the two-dimensional setting, in the case where $\Omega$ is a ball. 
We will prove that there exists a unique minimizer, which is  anti-symmetric, for $p=2$ and sufficiently small $\theta$.
On the contrary, the minimizers are not anti-symmetric  for $p$ sufficiently large. 
This shows a symmetry breaking phenomenon, which generalizes
the results proved by Girao and Weth in the case $\theta=0$. 
Note that because of the difficulty given by the lack of coercivity of our functional, our technique is quite different from that one of \cite{Girao-Weth}.

\section{Existence of a minimizer}
In this section we prove the existence
 of a minimizer for problem \eqref{p} by adapting the technique of \cite{DGS}. We will also make use of an estimate proved in  \cite{BO} (see also \cite{ABFOT} and \cite{BDO}).
\begin{theorem}\label{existence} 
Under the assumptions \eqref{teta}-\eqref{qpa}, \eqref{F0}-\eqref{condF},
there exists a minimizer $u$ which realizes $\lambda^{\theta,p}(\Omega)$, as defined in \eqref{p}.
\end{theorem}
\begin{proof}
We first observe that the growth assumption \eqref{growth} on
$F$ and the condition $\|u\|_{L^{p}(\Omega)}=1$ in the functional
imply that $\lambda^{\theta,p}\in\mathbb{R}$. For any fixed  $n\in \N$, let us define
\begin{gather*}
\displaystyle H_{n}(v)=\int_{\Omega}\frac{|\nabla v|^2-F(|x|,v)}{(1+|v|)^{2\theta}}\, dx
-\left(\lambda^{\theta,p}(\Omega)+\frac 1n\right),
\\
\displaystyle H_{\infty}(v)=\int_{\Omega}\frac{|\nabla v|^2-F(|x|,v)}{(1+|v|)^{2\theta}}\, dx
-\lambda^{\theta,p}(\Omega)
\end{gather*}
for any $v\in W^{1,q}(\Omega)$ such that $ v\neq 0,
\norm{v}_{L^p(\Omega)}=1$ and $\displaystyle \int_\Omega v=0$. By the definition of infimum, for any fixed $n\in \N$ there exists $u_n\in W^{1,q}(\Omega)$, $ u_n\neq 0,$
such that
\begin{equation}\label{Hn}
\norm{u_n}_{L^p(\Omega)}=1, \,\,\,\,\,\,\,\,\,\,\,\,\,\,\displaystyle
\int_{\Omega}u_n \,dx=0\,,\,\,\,\,\,\,\,\,\,\,\,\,H_{n}(u_n)<0\,.
\end{equation}
Now, by the growth assumption \eqref{growth} on $F$, since the functions $u_n$ have $L^p-$norm equal to 1,  we have 
\begin{equation}\label{Fn}
\int_{\Omega} \frac{|F(|x|,u_n)|}{(1+|u_n|)^{2\theta}}\, dx\le C\,,
\end{equation}
where $C$ is a positive constant which does not depend on $n$.  
 
 \noindent From now on we will denote by $C$ a positive constant which depends on the data and which can vary from line to line.

\noindent Since $H_{n}(u_n)<0$,  estimates (\ref{Hn}) and (\ref{Fn}) imply that
\begin{equation}\label{stima1}
\int_{\Omega}\frac{|\nabla u_n|^2}{(1+|u_n|)^{2\theta}}\, dx\le C\,.
\end{equation}
Now we  prove that $|\nabla u_n|$ is bounded in $L^{q}(\Omega)$, that is, for any $n\in \N$,
 \begin{equation}\label{stimagrad}
\|\nabla u_n\|_{L^{q}(\Omega)}\le C.
\end{equation}

\noindent We adapt the estimate used in    Theorem 2.1 of \cite{BO} and we distinguish the case where $N\ge 3$ and the case $N=2$.

\noindent Let $N\geq 3$ with $q=\frac{2N(1-\theta)}{N-2\theta}$.
 We begin by applying   the H\"older inequality since $q<2$; then we use
 estimate \eqref{stima1} and,  since the mean value of $u_n$ is null,
 by the  Sobolev inequality, we get
\begin{eqnarray*}
\int_\Omega |\nabla u_n|^q\, dx\le	&&  \left (\int_\Omega   \frac{|\nabla u_n|^2}{(1+|u_n|)^{2\theta}}
	 \, dx  \right )^\frac{q}{2}
	\left ( \int_\Omega (1+|u_n|)^{\frac{2\theta q}{2-q}}  \, dx    \right)^{1-\frac q2}\\
\notag&\\
\notag	\le &&
C \left (\int_\Omega   \frac{|\nabla u_n|^2}{(1+|u_n|)^{2\theta}}
	 \, dx  \right )^\frac{q}{2}
	\left (1+ \int_\Omega |u_n|^{q^*}  \, dx    \right)^{1-\frac q2}\\
	\notag&\\
\notag	\le &&
C \left (1+ \int_\Omega |\nabla u_n|^{q}  \, dx    \right)^{\frac{q^*}{q}(1-\frac q2)}
\end{eqnarray*}
where we have used the equality $\frac{2\theta q}{2-q}=q^*$. Since $N\ge 3$, we deduce that $\frac{q^*}{q}(1-\frac q2)<1$ and
\eqref{stimagrad} is proved.

\noindent Let $N=2$. Similarly to above, by using the H\"older inequality, 
estimate \eqref{stima1}, the inclusion  $L^{\frac{2 q}{2-q}}(\Omega)\subset
L^{\frac{2 q\theta}{2-q}}(\Omega)$ and the Sobolev inequality, we get
\begin{eqnarray*}
\int_\Omega |\nabla u_n|^q\, dx\le	&&  \left (\int_\Omega   \frac{|\nabla u_n|^2}{(1+|u_n|)^{2\theta}}
	 \, dx  \right )^\frac{q}{2}
	\left ( \int_\Omega (1+|u_n|)^{\frac{2\theta q}{2-q}}  \, dx    \right)^{1-\frac q2}\\
\notag&\\
\notag	\le &&C
\left (\int_\Omega   \frac{|\nabla u_n|^2}{(1+|u_n|)^{2\theta}}
	 \, dx  \right )^\frac{q}{2}
	\left (1+ \int_\Omega |u_n|^{\frac{2q}{2-q}}  \, dx    \right)^{(1-\frac q2)\theta}\\
	\notag&\\
\notag	\le &&
C \left (1+ \int_\Omega |\nabla u_n|^{q}  \, dx    \right)^{\theta}.
\end{eqnarray*}
Since $\theta<1$, \eqref{stimagrad} follows again.

\noindent By  the Poincar\'e-Wirtinger inequality, since the mean value of $u_n$ is zero,
we deduce by \eqref{stimagrad} that 
\begin{equation}\label{unboudedinWq}
u_n \quad\hbox{is bounded in
   $W^{1,q}(\Omega)$}
\end{equation}
and therefore there exists a function $u\in W^{1,q}(\Omega)$ such that, as $n$ goes to $\infty$, up to a subsequence,
\begin{equation}
   \label{convg}
u_n   \longrightarrow u \quad\hbox{in
   $W^{1,q}(\Omega)$ weakly, }
\end{equation}
\begin{equation}
   \label{convu}
u_n\longrightarrow u \quad\hbox{in
   $L^{r}(\Omega)$, $\quad 1\le r<q^*$\,,}
\end{equation}
\begin{equation}
   \label{convqo}
   u_n\longrightarrow u \quad\hbox{a.e. in $\Omega$}.
\end{equation}

\noindent Let $N\geq 3$ with $q=\frac{2N(1-\theta)}{N-2\theta}$. Note that (\ref{convu}), since $p<q^*$, implies that 
\[
\displaystyle \int_{\Omega}u dx=0\,\,,\,\,\,\,\,\,\,\,\,\norm{u}_{L^p(\Omega)}=1.
\]
We claim that
\begin{equation}\label{Hinfty}
H_{\infty}(u)\leq 0.
\end{equation}
Let
\begin{equation}\label{defnPsi}
 \Psi (\xi ) :=  
 \int_0 ^{\xi } (1+|t|)^{-\theta } dt = 
 \frac{\mbox{sgn}\, \xi }{1-\theta } [ (1+|\xi |)^{1-\theta } -1],\qquad \xi\in \R
\end{equation}
and observe that, by \eqref{stima1}, 
$$
\int_\Omega |\nabla \Psi(u_n)|^2 \,dx= \int_{\Omega}\frac{|\nabla u_n|^2}{(1+|u_n|)^{2\theta}}\, dx\le C.
$$
Moreover, $\Psi(u_n)$ is bounded in $L^2(\Omega)$,  since $2(1-\theta)\leq q$ and $u_n$ is bounded in $L^q(\Omega)$ by (\ref{unboudedinWq}).
We infer that $\Psi(u_n)$ is bounded in $W^{1,2}(\Omega)$ and, up to a subsequence,
$
\Psi(u_n)$ converges
 weakly in
  $W^{1,2}(\Omega)$ to a limit which is necessarily $
\Psi(u)$, by (\ref{convqo}).
Therefore, by the weak semi-continuity of the norm and  inequality in  (\ref{Hn}),  as  
$n$ goes to $\infty$, up to a subsequence,
$$
\|\Psi(u)\|^2_{W^{1,2}}\le \liminf_{n\to +\infty} \|\Psi(u_n)\|^2_{W^{1,2}}\le 
\lim_{n\to +\infty} \int _\Omega \frac{
F(|x|,u_n)}{(1+|u_n|)^{2\theta}}\, dx
+\lim_{n\to +\infty}\left(\lambda^{\theta,p}(\Omega)+\frac 1n\right)
\,.
$$
To pass to the limit in the first term of the right hand side, one can use the Lebesgue theorem. Indeed, the pointwise convergence is given by (\ref{convqo}). 
The growth assumptions \eqref{growth} on $F$ and \eqref{convu}, since $p< q^*$, imply the existence of a function $h\in L^p(\Omega)$
such that 
\[
\frac{
|F(|x|,u_n)|}{(1+|u_n|)^{2\theta}}\leq h(x) \qquad \hbox{ a.e. in }\Omega\,.
\]
Finally we get
\[
\int_{\Omega}\frac{|\nabla u|^2}{(1+|u|)^{2\theta}}\, dx\le \int _\Omega \frac{
F(|x|,u)}{(1+|u|)^{2\theta}}\, dx+\ \lambda^{\theta,p}(\Omega) \,,
\]
that is, \eqref{Hinfty} holds.
By the definition of $\lambda^{\theta,p}(\Omega)$, necessarily we have
$
H_{\infty}(u)=0.
$
We observe that $\Psi(u)\neq 0$, since $\norm{u}_{L^p(\Omega)}=1$.
\par
It remains to conclude the proof in the case $N=2$. Indeed when $N=2$
 we have $1<p<+\infty$ (see \eqref{qpa}) and $2(1-\theta)\leq q <2$
 (see \eqref{q2}), so that the convergences \eqref{convg},
 \eqref{convu} and \eqref{convqo} do not imply, in general, that $\|u\|_{L^{p}(\Omega)}=1$.
However in view of \eqref{stima1} and since $2(1-\theta)\leq q$ we
obtain that $\Psi(u_{n})$ is bounded in $W^{1,2}(\Omega)$. From the
Sobolev embedding theorem it follows that $\Psi(u_{n})$ is bounded in
$L^{r}(\Omega)$ for any $1\leq r<+\infty$. Since $\Psi(\xi)$ growths
like $|\xi|^{1-\theta}$ with $1-\theta>0$, we conclude that $u_{n}$ is
bounded in $L^{r}(\Omega)$ for any $1\leq r<+\infty$. We
obtain that $\|u\|_{L^{p}(\Omega)}=1$ and the arguments developed in
the case $N=3$ allow us to conclude that $H_{\infty}(u)=0$.
\end{proof}



\section{Identification of symmetry}\label{section_definition_symmetry}

In this section we generalize some known symmetry criteria (cf. \cite{BS}). 
We first introduce some notation and definitions.
Let $\Omega $ be a domain that is radially symmetric w.r.t.  the origin.
In other words, $\Omega $ is either an annulus, a ball, or the exterior of a ball in $\mathbb{R} ^N $.
If $u: \Omega \to \mathbb{R} $  is a measurable function, 
we will for convenience always extend $u$ onto $\mathbb{R} ^N $ by setting $u(x)=0 $ for $x\in \mathbb{R} ^N \setminus \Omega  $.

\begin{definition}
Let 
${\mathcal H} _0 $ be the family of open half-spaces $H$ in $\mathbb{R} ^N $ such that $0\in \partial H $. 
For any  $H \in {\mathcal H} _0 $, let $\sigma _H $ denote the reflection in $\partial H $. 
We write
$$
\sigma _H u (x) := u(\sigma _H x), \quad x\in \mathbb{R} ^N .
$$
The two-point rearrangement w.r.t. $H$  is given by
\[
u_H (x) := 
\begin{cases}
\max \{ u(x); u(\sigma _H x ) \} & \mbox{ if }\ x \in H ,
\\
\min \{ u(x); u(\sigma _H x ) \} & \mbox{ if } \ x\not\in H .
\end{cases}
\]  
\end{definition}
The notion of two-point rearrangement  was introduced more than fifty years ago as a set  transformation 
in \cite{W}, and was applied to variational problems  for the first time by Brock and Solynin in \cite{BS}.

Note that one has $u= u_H $ iff $u(x)\geq u(\sigma _H x )$ for all $x\in H$.
Similarly, $\sigma _H u= u_H $ iff $u(x)\leq u(\sigma _H x) $ for all $x\in H$.
\\
We will make use of the following properties of the two-point rearrangement (see \cite{BS}).

\begin{lemma}\label{1}
 Let $H\in {\mathcal H} _0 $.

\begin{enumerate}
\item
If $A\in C([0, +\infty ),\mathbb{R})$, $u:\Omega \to \mathbb{R} $ is measurable and $A (|x|,u)\in L^1(\Omega)$, then  $A (|x|, u_H)\in L^1(\Omega)$ and 
\begin{equation}
\label{identity1}
\int_\Omega A (|x|, u)\, dx =\int_\Omega A (|x|, u_H)\, dx \, .
\end{equation}
\item If $B\in L^{\infty } (\mathbb{R})$, $ u\in W^{1,p}(\Omega)$ for some $p\in [1, +\infty )$, then 
\begin{equation}
\label{identity2}
\int_\Omega B(u) |\nabla u|^p \, dx =\int_\Omega B(u_H ) |\nabla u_H |^p
\, dx 
\end{equation}
\end{enumerate}
\end{lemma}

\begin{proof}Since $|\sigma _H x|= |x|$, we have for a.e. $x\in H\cap \Omega $, 
\begin{equation*}
 A(|x|, u(x))+ A(|\sigma _H x| ,u(\sigma _H x))= 
A(|x|, u_H (x))+ A(|\sigma _H x| ,u_H (\sigma _H x)),
\end{equation*}
and
\begin{equation*}
\begin{split}
  B(u(x))|\nabla u(x)|^p + B(u(\sigma _H x))|\nabla u(\sigma _H x)|^p = &
B(u_H (x))|\nabla u_H (x)|^p \\
& {} + B(u_H (\sigma _H x))|\nabla u_H (\sigma _H x)|^p .
\end{split}
\end{equation*}
Now (\ref{identity1}) and (\ref{identity2})
follow from this by integration over $H\cap \Omega$. 
\end{proof}

In order to study the symmetry of minimizers of \eqref{p} we introduce the notion of foliated Schwarz symmetrization of a function, a function which is   axially symmetric with respect to an axis passing through the
origin and nonincreasing in the polar angle from this axis.

\begin{definition}
If $u:\Omega \to \mathbb{R} $ is measurable, the {\sl foliated Schwarz symmetrization } $u^* $ of $u$ is defined as  the (unique) function satisfying the following properties:
\begin{enumerate}
\item
there is a function $w : [0, +\infty ) \times [0, \pi ) \to \mathbb{R } $, 
$w= w (r, \theta )$, which is nonincreasing in 
$\theta $, and 
\[
u^* (x) = w\left( |x| , \arccos (x_1 /|x| ) \right) , \quad (x \in \Omega );
\]  
\item
$
{\mathcal L} ^{N-1}  \{ x:\, a < u(x) \leq b,\, |x| =r \} = {\mathcal L} ^{N-1} \{ x:\, a < u^* (x) \leq b ,\, |x|=r \}$  
for all $a,b \in \mathbb{R}$ with $a<b$, and $r\geq 0$.
\end{enumerate}
\end{definition} 
\begin{definition}
Let $P_N$ denote the point $(1,0, \ldots , 0)$, the 'north pole' of
the unit sphere ${\mathcal S} ^{N-1} $. 
We say that $u$ is {\sl foliated Schwarz symmetric  w.r.t.  $P_N$} if 
$u=u^*$ - that is, $u$ depends solely on $r$ and on $\theta $ - the 'geographical width' -, 
and is nonincreasing in $\theta $. 

We also say that $u$ is {\sl foliated Schwarz symmetric w.r.t. a point $P\in {\mathcal S}^{N-1} $ } 
if there is a rotation about the origin $\rho $ such that $\rho (P_N)= P$, and 
$u(\rho (\cdot)) = u^* (\cdot )$.
\end{definition}
In other words, a function $u:\Omega\to \R$ is foliated Schwarz symmetric with respect to $P$ if, for every $r>0$ and $c\in \R$,
the restricted superlevel set $\{x: |x|=r, u(x)\geq c\}$ is equal to $\{x: |x|=r\}$ or a geodesic ball in the
sphere $\{x: |x|=r\}$ centered at $rP$. In particular $u$ is axially symmetric with respect to the axis $\R P$.

\noindent Moreover
a measurable function  $u:\Omega \to \mathbb{R}$ is foliated Schwarz symmetric w.r.t. 
$P\in {\mathcal S}^{N-1} $ iff $u=u_H $ for all $H\in {\mathcal H} _0 $ with $P\in H$.

The main result of this section is  the following result which gives a tool to establish if a measurable function is foliated Schwarz symmetric with respect to some point $P$.

\begin{theorem} 
\label{teofr}
 Let $u\in L^{p} (\Omega )$ for some $p\in [1, +\infty )$, 
and assume that for every $H\in {\mathcal H} _0 $ one has either $u= u_H $, or $\sigma _H u= u_H $.
Then $u$ is foliated Schwarz symmetric w.r.t. some point $P\in {\mathcal S}^{N-1} $.
\end{theorem}
Note that the above result has been shown for continuous functions by Weth in \cite{Weth}. 

\begin{theorem} 
\label{theoweth}
 Let $u\in C(\R^N)$ 
and assume that for every $H\in {\mathcal H} _0 $ one has either $u= u_H $, or $\sigma _H u= u_H $.
Then $u$ is foliated Schwarz symmetric w.r.t. some point $P\in {\mathcal S}^{N-1} $.
\end{theorem}
The idea in our proof is to use an approximation argument.  
Let $\varphi \in C^{\infty } _0 (\mathbb{R}^N )$, $\varphi \geq 0$, with $\displaystyle \int _{\mathbb{R}^N } \varphi (x)\, dx =1 $. 
Moreover, assume that $\varphi $ is radial and radially non increasing, that is, 
there is a nonincreasing function $h:[0, +\infty ) \to [0, +\infty ) $ 
such that $\varphi (x)= h(|x|) $ for all $x\in \mathbb{R} ^N $.  
For any function $\varepsilon >0 $, define $\varphi _{\varepsilon } $ by
$\varphi _{\varepsilon } (x) := \varepsilon ^{-N} \varphi ( \varepsilon ^{-1} x)$,  ($x\in \mathbb{R}^N $).
For any $u \in L^1 _{{loc}} (\R^N )$ let 
$u_{\varepsilon } $ be  the standard  mollifier of $u$, given by
$$
u_{\varepsilon} (x) := (u* \varphi _{\varepsilon }) (x) \equiv \int_{\mathbb{R} ^N } u(y) \varphi _{\varepsilon } (x-y) \, dy , \quad (x\in \mathbb{R} ^N ).
$$
The following property is crucial. It allows a reduction to $C^{\infty } $-functions.

\begin{lemma} \label{cruc} 
Let $u\in L^{p} ( \mathbb{R} ^N )$ for some $p\in [1, +\infty ) $, and let $H\in {\mathcal H} _0$ such that $u= u_H $. Then $u_{\varepsilon }= (u_{\varepsilon })_H $ for every $\varepsilon >0$.
\end{lemma}

\begin{proof}
It is easy to see that
$$
|x-y|=|\sigma _H x -\sigma _H y | \leq |\sigma _H x -y | = 
|x-\sigma _H y |, 
$$ 
whenever $x,y \in H$. Since $u(y) \geq u(\sigma _H y)$ and 
 since $\varphi _{\varepsilon } $ is radial and 
radially nonincreasing, we have for every $x\in H$, 
\begin{eqnarray*}
 & & u_{\varepsilon } (x) - u_{\varepsilon } (\sigma _H x)  =  \int _{\mathbb{R}^N } 
u(y) [ \varphi _{\varepsilon }(x-y ) - \varphi _{\varepsilon }(\sigma _H x -y ) ] \, dx 
\\
 & = & \int _H 
\left\{ u(y) [ \varphi _{\varepsilon }(x-y ) - \varphi _{\varepsilon } 
(\sigma _H x -y ) ] + u(\sigma _H y ) [ \varphi _{\varepsilon } 
( x-\sigma _H y ) - \varphi _{\varepsilon } (\sigma _H x -\sigma _H y )]     \right\} \, dx
\\
 & = & \int _H 
( u(y)-u(\sigma _H y) )  [ \varphi _{\varepsilon }(x-y ) - \varphi _{\varepsilon }(\sigma _H x -y ) ]  \, dx \geq 0.
\end{eqnarray*}
The Lemma is proved.
\end{proof}

\begin{corollary}\label{cruccorollary}
 Let $u\in L^{p} ( \mathbb{R} ^N )$ for some $p\in [1, +\infty ) $, 
and let $H\in {\mathcal H}_0 $ such that $\sigma _H u= u_H $. 
Then $\sigma _H (u_{\varepsilon })= (u_{\varepsilon })_H $ for every $\varepsilon >0$.
\end{corollary}

We are now able to prove Theorem \ref{teofr}.
\medskip

\begin{proof}[Proof of Theorem \ref{teofr}]

  Since  
for every $H\in {\mathcal H} _0 $ one has either $u= u_H $, or $\sigma _H u= u_H$,
Lemma \ref{cruc} and Corollary \ref{cruccorollary} apply. 
Then
either $u_\varepsilon= (u_\varepsilon)_H $, or $\sigma _H u_\varepsilon= (u_\varepsilon)_H$ for every $\varepsilon>0$.
Since $u_{\varepsilon } \in C^{\infty } (\mathbb{R} ^N)$, 
Theorem \ref{theoweth} tells us that $u_{\varepsilon } $ is foliated Schwarz symmetric w.r.t. some point 
$P_{\varepsilon } \in {\mathcal S} ^{N-1} $, for every $\varepsilon >0 $. Since ${\mathcal S} ^{N-1} $ is compact, 
there is a sequence of positive numbers $\{\varepsilon _n \} $ and a point 
$P\in {\mathcal S} ^{N-1} $ such that $u_{\varepsilon _n } $ is foliated Schwarz symmetric w.r.t. a point 
$P_n  \in {\mathcal S} ^{N-1} $ and $\varepsilon _n \to 0$, $P_n  \to P$ as $n\to + \infty $. 
Let $\rho _n $ and $\rho $ be rotations such that $\rho _n (N ) = P_n $, 
($n\in \mathbb{N} $), and $\rho (N)=P$. Writing $u_n := u_{\varepsilon _n }$ we have that 
\begin{equation}
\label{conv}
u_n  (\rho _n (\cdot )) = (u_n ) ^* (\cdot ) , \quad (n\in \mathbb{N} ).
\end{equation} 
Since $u_n \to u$, it follows that 
$(u_n ) ^* \to u^* $ in $L^p (\mathbb{R} ^N )$, and since $P_n \to P$  
we also have that $u_n (\rho _n (\cdot)) \to u(\rho (\cdot )) $ in $L^p (\mathbb{R}^N )$, as  $n\to \infty $. 
This, together with  (\ref{conv}) implies that $u(\rho (\cdot )) = u^* (\cdot )$. The Theorem is proved.     
\end{proof}


\section{Symmetry of minimizers}\label{section_foliated_symmetry}
In this section we study the properties of symmetry of minimizers of \eqref{p}. The main result is the following 

\begin{theorem}\label{thm_foliated_Schwarz_symmetric}
Assume   (\ref{growth}), (\ref{growthFt}),
and (\ref{condF})  if $p\in (1,2)$.
Then every minimizer of \eqref{p} is foliated Schwarz symmetric w.r.t. some point $P\in {\mathcal S}^{N-1}$.
\end{theorem}

\begin{remark}\rm  Condition (\ref{condF}) is equivalent to
\begin{equation}
\label{tddtF}
t\frac{\partial}{\partial t} \left( \frac{F(r,t)}{(1+|t|)^{2\theta } } 
\right) \leq 0 \quad \forall (r,t) \in [0,+\infty)\times \mathbb{R} .
\end{equation}
It is satisfied, for instance, if $F(r,t)= F(r,-t)$ and if 
\begin{equation}
\label{condF1}
(1+t)F_t (r,t) -2\theta F(r,t) \leq 0 \quad \mbox{for $t\geq 0$}.
\end{equation}
An example  is  
$$
F(r,t )= -c_0 |t|^{\alpha } , \quad \alpha \geq 2\theta , \, c_0 \geq 0.
$$     
Observe that $F$ satisfies the growth condition \eqref{growth} with suitable $c_0$ and $\alpha$ such that $2\theta\le \alpha\le p$. 
 \end{remark}
 \medskip
\begin{proof}  
We divide the proof into four steps.

\noindent {\sl Step 1}
Let $H\in {\mathcal H}_0$, and let $u$ be a minimizer of \eqref{p}.
The Euler equation satisfied by $u$ is
\begin{eqnarray}
\label{euler1}
 & & -\nabla \left( 
 \frac{\nabla u}{(1+|u|)^{2\theta} }  
 \right)
 -\theta 
\frac{|\nabla u|^2  \mbox{sgn}\, u}{(1+|u|)^{2\theta +1 } } 
+c +d |u|^{p-2} u = g(|x|,u) \quad \mbox{in $\Omega $ },
\\
\label{bdry1}
 & & \frac{\partial u}{\partial \nu } = 0 \quad \mbox{on $\partial \Omega $},
 \end{eqnarray}
 where $c,d\in \mathbb{R}$,
\begin{equation}
\label{defg}
g(r,t) := \frac{\partial}{\partial t} \left( \frac{F(r,t)}{2(1+|t|)^{2\theta }} \right) , \qquad \forall (r,t)\in [0,+\infty )\times \mathbb{R}
\end{equation}
and $\nu $ denotes the exterior unit normal to $\Omega $.
Setting 
$$
I (v):= \int_{\Omega}\frac{|\nabla v|^2-F(|x|,v)}{(1+|v|)^{2\theta}} \, dx,
$$ 
we have, by Lemma \ref{1}, 
$$
u_H\ne 0, \quad u_H\in W^{1,q} (\Omega) ,\quad \int_\Omega u_H\, dx=0 , \quad \|u_H\|_{L^p}=1, \quad I(u)= I(u_H ).
$$
Hence, $u_H $ is a minimizer, too, so that it satisfies  
\begin{eqnarray}
\label{euler2}
 & & -\nabla \left( 
 \frac{\nabla u_H}{(1+|u_H|)^{2\theta} }  
 \right)
 -\theta 
\frac{|\nabla u_H|^2  \mbox{sgn}\, u_H}{(1+|u_H|)^{2\theta +1 } } 
+c' +d' |u_H|^{p-2} u_H = g(|x|,u_H) \quad \mbox{in $\Omega $ },
\\
\label{bdry2}
 & & \frac{\partial u_H}{\partial \nu } = 0 \quad \mbox{on $\partial \Omega$,}
 \end{eqnarray}
 where $c',d'\in \mathbb{R}$.

\noindent {\sl Step 2}
We claim that  $u,u_H \in W^{1,q} (\Omega )\cap W^{2,2} (\Omega)\cap C^1 (\overline{\Omega}) $. 
Set
$$
\Phi (\eta ) := \Psi^{-1} (\eta )
$$
where $\Psi$ has been defined in (\ref{defnPsi}). Let
$U:= \Psi (u)$.
Note that $u= \Phi (U)$, $u_H = \Phi (U_H) $,  
$$
\Phi (\eta ) = \left( [1+ (1-\theta) |\eta | ]^{1/(1-\theta )} -1 \right) \mbox{sgn}\, \eta ,
$$
and $\Phi $ is locally Lipschitz continuous.
Rewriting (\ref{euler1}) and (\ref{euler2}) in terms of $U$ and $U_H$ we find
\begin{eqnarray}
\label{euler3}
 & & 
 -\Delta U + d M(U) = N(|x|, U),
 \\
\label{euler4}
 & &
 -\Delta U_H + d M(U_H ) = N(|x|, U_H )
 \end{eqnarray}
 in $\Omega $, where
\begin{eqnarray*}
M(t) & := & |\Phi (t)|^{p-2} \Phi (t) (1+|\Phi (t)|)^{\theta },
\\
N(r,t) & := & (g(r,t)-c ) (1+|\Phi (t)|)^{\theta } \, ,
\end{eqnarray*}
for any $r\in [0,+\infty )$, $t\in \mathbb{R}$.

\noindent  Observe that, by the growth conditions  \eqref{growth},  \eqref{growthFt} and definition of $ \Phi (t) $,  we have
\begin{eqnarray*}
|g(r,t) |& \le & c_\theta(1+ |t|)^{p-1-2\theta }, \\
|M(t) |& \le & c'_\theta(1+ |t|)^{\frac{p-1+\theta}{1-\theta} },
\\
|N(r,t) |& \le & c''_\theta(1+ |t|)^{p-1-2\theta+\frac{\theta}{1-\theta} }.
\end{eqnarray*}
Now, the growths of $M$ and $N$ allow us to to apply  classical techniques for Neumann problems (see p.  272 of \cite{mawhin} and p. 271 of \cite{S}) to 
state that
$U\in H^1(\Omega)$ is in fact $C^{1,\beta } (\overline{\Omega })$, with  $\beta \in (0,1)$. Therefore $u$ has the same regularity.

\noindent {\sl Step 3}
Integrating (\ref{euler1}) and (\ref{euler2}) give
\begin{eqnarray}
\label{integral1}
 & & -\theta \int_{\Omega} \frac{|\nabla u|^2 \mbox{sgn}\, u}{(1+|u|)^{2\theta} }\, dx  +c \int_{\Omega } dx + d \int_{\Omega } |u|^{p-2} u\, dx = 
 \int_{\Omega } g(|x|,u )\, dx,
 \\
\label{integral2}   
 & & -\theta \int_{\Omega} \frac{|\nabla u_H|^2 \mbox{sgn}\, u_H}{(1+|u_H|)^{2\theta} }\, dx  +c' \int_{\Omega } dx + d' \int_{\Omega } |u_H|^{p-2} u_H\, dx = \int_{\Omega } g(|x|,u_H )\, dx.
\end{eqnarray}
Further, multiplying  (\ref{euler1}) and (\ref{euler2}) with $u$ and $u_H $ respectively, then integrating and using the constraints, yield
\begin{eqnarray}
\label{integral3}
 & &  \int_{\Omega} \frac{|\nabla u|^2 [1+ (1-\theta ) |u|]}{(1+|u|)^{2\theta +1} }\, dx  + d  = 
 \int_{\Omega } ug(|x|,u )\, dx,
 \\
\label{integral4}
 & &  \int_{\Omega} \frac{|\nabla u_H |^2 [1+ (1-\theta ) |u_H|]}{(1+|u_H |)^{2\theta +1} }\, dx  + d'  = 
 \int_{\Omega } u_H g(|x|,u_H )\, dx
.
\end{eqnarray}
Now (\ref{integral1})--(\ref{integral4}) together with Lemma \ref{1} show that necessarily 
$$
c=c'\,,\,\,\,\,\,\,\,\,\,\,d=d'\,.
$$
Moreover, if $p\in(1,2)$, then (\ref{tddtF}) holds, so that (\ref{integral3}) yields $d\leq 0$.
\\   

\noindent {\sl Step 4}

Note that $t\to M(t)$ is nondecreasing.
Set $h:=U_H-U$, and note that $h\geq 0 $ in $\Omega \cap H$. Subtracting (\ref{euler4}) from (\ref{euler3}), we split into two cases.
\begin{enumerate} 
\item
Let $p\geq 2$. Then we find that 
$$
-\Delta h= L (x) h \quad \mbox{in $\Omega \cap H$,}
$$
where 
$$
L (x) := \left\{
\begin{array}{ll}
\frac{ 
N(|x|, U_H (x))-N (|x|,U(x)) -d [M(U_H (x))-M(U(x))] }{h(x)} 
 & \quad \mbox{if $h(x)>0$},
 \\
0 & \quad \mbox{if $h(x)=0$}
\end{array}
\right.   
$$ 
is a bounded function.
\item
Let $p \in (1,2)$. Then $d\leq 0$, so that
$$
-\Delta h \geq P(x)h \quad \mbox{in $\Omega \cap H$},
$$
where  
$$
P (x):= \left\{
\begin{array}{ll}
\frac{ 
N(|x|, U_H (x))-N (|x|,U(x))}{h(x)}   
 & \quad \mbox{if $h(x)>0$,}
 \\
0 & \quad \mbox{if $h(x)=0$}
\end{array}
\right.   
$$ 
is a bounded function. 
\end{enumerate}
Thus, in both cases the Strong Maximum Principle tells us that either $h(x)\equiv 0$, or $h(x)>0$ throughout $\Omega \cap H $.
This implies that we have either $u= u_H$, or $\sigma _H u =u_H $ in $\Omega $. 
 By Theorem \ref{teofr} we deduce that $u$ is foliated Schwarz symmetric.
\end{proof}


\section{Anti-symmetry for $p=2$ in dimension 2}
In this section we study symmetry properties of the solutions to \eqref{p} in the case 
$p=2$,  $\Omega=B$, where $B$ is a ball in $\mathbb{R}^2$, and $F\equiv 0$. 
We will show that for small parameter values $\theta$,
there exists a unique minimizer of
$$
v\to \int_{B}\frac{|\nabla v|^2}{(1+|v|)^{2\theta}}\, dx,\,\,\,\,\,\, v\in W^{1,q}(B), \, v\neq 0 \\, \int_{B}v\, dx=0, \, \norm{v}_{L^2(B)}=1\, ,
$$
which is anti-symmetric. 
Recall that  $\theta$ satisfies \eqref{teta} and $q$ satisfies \eqref{q2}. With abuse of notations, we will denote the infimum of the above
functional by $\lambda^{\theta,2}(B)$.

In the following we will use the notations of the proof of Theorem 4.1. More in details, let $ u_{\theta } $ be a minimizer for $\lambda^{\theta,2}(B)$, with corresponding constants $c=c_{\theta } $ and $d= d_{\theta }$, see equation (\ref{integral1}). 
By (4.12) 
we have, 
\begin{equation}\label{defn-d_theta}
d_\theta=
-\int_B\frac{|\nabla u_\theta|^2 [1+(1-\theta) |u_\theta|)]}{(1+|u_\theta|)^{2\theta+1}}\, dx\,.
\end{equation}
We will also frequently work with the functions 
$$
U_{\theta } := \Psi _{\theta } (u_{\theta } )
$$ 
where 
$$
\Psi _{\theta } (\xi ) = \frac{\mbox{sgn} (\xi) }{1-\theta } [(1+ |\xi | )^{1-\theta } -1 ]
$$ 
(see (\ref{defnPsi})),
and
$$
\Phi _{\theta }(\eta ) = \Psi _{\theta } ^{-1}(\eta ) = \left( [1+ (1-\theta) |\eta | ]^{1/(1-\theta )} -1 \right) \mbox{sgn}(\eta)\,.
$$
Our calculations will often contain a generic constant $C$ that may vary from line to line, but will be {\sl independent of } $\theta $.
  
Furthermore, as a consequence of Theorem \ref{thm_foliated_Schwarz_symmetric}, we will   assume that  $u_{\theta } $ is foliated Schwarz symmetric w.r.t. the positive $x_1 $-half axis, that is,
\begin{equation}
\label{u_theta_even}
u_\theta(x_1,x_2)=u_\theta(x_1,-x_2)\,.
\end{equation}
The anti-symmetry of $u_{\theta } $ then reads as 
$
u_\theta(x_1,x_2):= - u_\theta(-x_1,x_2)
$, if $\theta $ is small.
\begin{lemma}
\label{first_lemma} Under  assumptions \eqref{teta}, \eqref{q2},
the function $\theta\to \lambda^{\theta,2}(B)$ is decreasing. Moreover
$\lambda^{\theta,2}(B)\leq \lambda^{2}(B)$, where
$$
\lambda^{2}(B)
=\inf\left\{\int_B |\nabla u|^2\, dx, u\in H^1(B), \int_B u\, dx=0, \norm{u}_{L^2(B)}=1\right\}\,.
$$
\end{lemma}
\begin{proof}
Let $\theta_1<\theta_2$, let $u_{\theta_1}$ be a minimizer for $\lambda^{\theta_1,2}(B)$, and let $ 2(1-\theta_1)\le q<2$. We obtain
$$
\lambda^{\theta_1,2}(B)=\int_B\frac{|\nabla u_1|^2}{(1+|u_1|)^{2 \theta_1}}\, dx
\geq
\int_B\frac{|\nabla u_1|^2}{(1+|u_1|)^{2 \theta_2}}\, dx
\geq 
\lambda^{\theta_2,2}(B).
$$
Next, let $u$ be a minimizer for $\lambda^{2}(B)$. Then
$$
\lambda^{2}(B)\geq \int_B |\nabla u|^2\, dx\geq 
\int_B\frac{|\nabla u|^2}{(1+|u|)^{2 \theta}}\, dx\geq \lambda^{\theta,2}(B)\,.
$$\end{proof}
\begin{lemma}
\label{lemma_d_theta} 
Under  assumptions \eqref{teta}, \eqref{q2}, let $u_\theta$ be a minimizer for $\lambda^{\theta,2}(B)$. Let $d_\theta$ be defined by (\ref{defn-d_theta}).
There holds
$\lim\limits_{\theta\to 0} d_{\theta}=-\lambda^{2}(B)$.
\end{lemma}
\begin{proof}
First we observe that
\begin{equation}
\label{d_above}
-d_\theta
=
\int_B\frac{|\nabla u_\theta|^2 (1+(1-\theta) |u_\theta|))}{(1+|u_\theta|)^{2\theta+1}}\, dx
\leq 
\int_B\frac{|\nabla u_\theta|^2}{(1+|u_\theta|)^{2\theta}}\, dx
\leq \lambda^{\theta,2}(B)
\leq \lambda^{2}(B)
\end{equation}
by Lemma \ref{first_lemma}.
On the other hand,
\begin{equation}
\label{ineq_d_theta}
-d_\theta
\geq 
\int_B \frac{|\nabla u_\theta|^2(1-\theta)}{(1+|u_\theta|)^{2\theta}}\, dx
=(1-\theta)
\int_B|\nabla U_\theta|^2\, dx\,.
\end{equation}
Moreover, it is easy to see that
$\norm{U_\theta}_{L^2(B)}$ is uniformly bounded, 
since $\norm{u_\theta}_{L^2(B)}=1$ and $\theta \leq \frac{1}{2}$.
Therefore also $\norm{U_\theta}_{H^1(B)}$ is uniformly bounded.
By compactness, as $\theta\to 0$, $U_\theta$ converges weakly to some function $V\in H^1(B)$ and strongly in $L^2(B)$. 
By the lower semi-continuity of the norm we get from (\ref{ineq_d_theta})
\begin{equation}
\label{d_below}
 \liminf_{\theta\to 0} (-d_\theta ) \geq 
 \norm{\nabla V}^2_{L^2(B)}\,.
\end{equation}
Further, the a.e. limit of $U_{\theta }$ is the limit of $u_\theta$, say $u$. By the uniqueness of the limit, $u= V$ a.e. in $B$. 
We recall that
$\norm{u_{\theta}}_{L^2(B)}=1$ and $\displaystyle \int_B u_{\theta }\, dx =0$.
Since $\norm{\Psi_{\theta}(u_\theta)}_{H^1(B)}$, by the growth of $\Psi_{\theta}$, we deduce that $u_{\theta } \to u$ in $L^2 (B)$ and that $\norm{V}_{L^2(B)}=1$ and 
$\displaystyle \int_B V\, dx=0$. Together with (\ref{d_below}), this implies
that
\begin{equation}
\label{d_below2}
 \liminf_{\theta\to 0} (-d_\theta)\geq \lambda^{2}(B)\,.
\end{equation}
Now the Lemma follows from inequalities (\ref{d_above}) and (\ref{d_below2}). 
\end{proof}

\begin{proposition}
\label{first_prop} Let $u_\theta$ be a minimizer for $\lambda^{\theta,2}(B)$. Under  assumptions \eqref{teta}, \eqref{q2},
we have 
\begin{enumerate}
\item
$\norm{u_\theta}_{W^{1,\infty}(B)} \leq C$, where $C$ does not depend on $\theta $.
\item
Let $u$ be the limit of $u_\theta$, as $\theta \to 0$,  in $W^{1,r}(\Omega)$, for every $r\in (1, +\infty )$. Then $\norm{u}_{L^2(B)}=1$ and $\displaystyle \int_B u\, dx=0$.
\end{enumerate}
\end{proposition}
\begin{proof}
The $H^1(B)$-norm of  ${U_\theta}$ is uniformly bounded by Lemma \ref{first_lemma}.
By multiplying equation (\ref{euler3}) by $U_{\theta}$ and integrating on $B$, one has that the right hand side of the equality is uniformly bounded, due to 
Lemma \ref{first_lemma}.
The growth of $M$ and Lemma \ref{lemma_d_theta} imply that
$\norm{U_\theta}_{L^{\frac{p-1+\theta}{1-\theta}}(B)}\leq C$.
This allows us to use the bootstrap argument described 
at p. 271 of \cite{S}.
\end{proof}

Let $v_\theta(x_1,x_2):= - u_\theta(-x_1,x_2)$. Then
we obtain from (\ref{integral1}),
\begin{equation}\label{defn-c_theta}
c_\theta=
\frac{\theta}{|B|}\int_B\frac{|\nabla u_\theta|^2 \mbox{sgn} (u_\theta)}{(1+|u_\theta|)^{2\theta+1}}\, dx
= - \frac{\theta}{|B|}\int_B\frac{|\nabla v_\theta|^2 \mbox{sgn} (v_\theta)}{(1+|v_\theta|)^{2\theta+1}}\, dx\,.
\end{equation}
\begin{lemma}\label{lemma_c_theta} 
Under  assumptions \eqref{teta}, \eqref{q2}, let $u_\theta$ be a minimizer for $\lambda^{\theta,2}(B)$. Let $c_\theta$ be defined by (\ref{defn-c_theta}).
There holds
$|c_{\theta}|
\leq 
C\theta \Vert u_{\theta }- v_{\theta }\Vert_{L^2(B)}$ for a positive constant $C$ independent on $\theta$.
\end{lemma}
\begin{proof}
By multiplying equation
 (4.3), (with $p=2$ and $g=0$), by $(1+|u_\theta|)^{\theta}$, we have
\begin{equation}\label{eqeulermultiplied}
-\nabla ( (1+|u_{\theta} |)^{-\theta } \nabla u_{\theta }) + c_{\theta } (1+|u_\theta |)^{\theta } + d_{\theta } u_{\theta }(1+|u_{\theta } |)^{\theta }=0.
\end{equation}
Integrating this gives
$$
c_{\theta } \int _B (1+|u_{\theta }|)^{\theta }\, dx + d_{\theta } \int _B u_{\theta} (1+|u_{\theta } |)^{\theta }\, dx=0,
$$
since $\frac{\partial u_{\theta}}{\partial \nu}=0$ on $\partial B$.
The first integral in this identity is bounded from below, and $|d_{\theta }|$ is bounded by Lemma \ref{lemma_d_theta}.
If $J$ denotes $\displaystyle \int _B u_\theta(1+|u_{\theta }|)^{\theta }\, dx$, we deduce that 
\begin{equation}\label{estimate_c_theta}
|c_{\theta}| \leq C |J|
\end{equation}
for a constant $C$ independent on $\theta$.
A change of variables gives
$$
J =\frac 12 \int_B [u_{\theta}(1+|u_{\theta}|)^{\theta}  -v_{\theta }(1 + |v_{\theta}|)^{\theta}]\, dx
\,.
$$
Since $\displaystyle \int _B u_{\theta}\, dx= \int_B v_{\theta}\, dx=0$, we get
$$
J= \frac 12 \int_B (u_{\theta}-v_{\theta}) [(1+ |v_{\theta} |)^{\theta }-1] \, dx+ \frac 12 \int_B u_{\theta}[(1+|u_{\theta}|)^{\theta} - (1+|v_{\theta}|)^{\theta}]\, dx\,.
$$
Let $J_1$ denote the first term and $J_2$ the second one in this identity.
A short computation shows that
$$
|J_1 | \leq \frac{\theta}{2} \int _B |u_{\theta}-v_{\theta}| |v_{\theta}|\, dx\,,
\,\,\,\,\,\,\,\,\,\,\,\,\,\,
|J_2 | \leq \frac{\theta}{2}  \int _B |u_\theta-v_{\theta}| |u_{\theta}|\, dx\,.
$$
Since $u_{\theta}$ and $v_{\theta}$ are uniformly bounded by Proposition \ref{first_prop}, this gives
$$
|J| \leq C \theta \Vert u_{\theta }-v_{\theta} \Vert_{L^2(B)} .
$$
The conclusion follows from estimate (\ref{estimate_c_theta}).
\end{proof}
Now we can prove the main result of the section
\begin{theorem}
\label{antsymm}

There is a number $\theta _0 >0$, such that every minimizer $u_{\theta } $ of $\lambda ^{\theta ,2} (B)$ satisfying  (\ref{u_theta_even}) is unique and anti-symmetric w.r.t. $x_1$, that is,
\begin{equation}
\label{antis}
u_{\theta}(x_1,x_2):= - u_{\theta}(-x_1,x_2) ,
\end{equation}  
for any $0<\theta < \theta _0 $.
\end{theorem}
\begin{proof}
We first prove that any minimizer is anti-symmetric.
Let 
$U_{\theta} := \Psi _{\theta } (u_{\theta})$ and 
$V_{\theta} :=\Psi _{\theta } (v_{\theta})$, where $v_{\theta}(x_1,x_2)=-u_{\theta}(-x_1,x_2)$.
Writing equation (\ref{eqeulermultiplied}) in terms of $U_{\theta}$ we have
$$
-\Delta U_{\theta} + c_{\theta} (1+ |u_{\theta}|)^{\theta } + d_{\theta } u_{\theta} (1+ |u_{\theta}|)^{\theta } =0\,.
$$
Similarly
$$
-\Delta V_{\theta} - c_{\theta} (1+ |v_{\theta}|)^{\theta} + d_{\theta } v_{\theta} (1+ |v_{\theta}|)^{\theta } =0.
$$
Subtract both equations from each other. Assuming that $U_{\theta}-V_{\theta}\not=0$ along a sequence $\theta \to 0$, 
we multiply by $\frac{U_{\theta}-V_{\theta}}{\Vert  U_{\theta}-V_{\theta}\Vert_{L ^2(B)}^{2}}$ and integrate.
Then we obtain
$$
\frac{\norm{\nabla (U_{\theta}-V_{\theta})} ^2_{L ^2(B)}}{\norm{U_{\theta}-V_{\theta}}^2_{L ^2(B)}} 
+
 \frac{c_{\theta}}{\Vert U_{\theta}-V_{\theta} \Vert_{L^2(B)}} \int_B [(1+ |u_{\theta}|)^{\theta} + 
 (1+ |v_{\theta}|)^{\theta}]
 \frac{U_{\theta}-V_{\theta}}{\Vert U_{\theta}-V_{\theta} \Vert_{L^2(B)}}\, dx=
$$
$$
=  - d_{\theta } \int_B [u_{\theta}) (1+ |u_{\theta}|)^{\theta } - v_{\theta} (1+ |v_{\theta})|)^{\theta}] 
\frac{U_{\theta}-V_{\theta}}{\Vert  U_{\theta}-V_{\theta} \Vert^2_{L^2(B)}} \, dx
\,.
$$
The second term of the left-hand side tends to zero, by Lemma \ref{lemma_c_theta},  and since
$
(1+ |u_{\theta}|)^{\theta} +(1+ |v_{\theta}|)^{\theta}
$ 
is uniformly bounded by Proposition \ref{first_prop}.
To estimate the  right-hand side, we first observe that $-d_{\theta } \to \lambda ^ 2 (B)$, by Lemma \ref{lemma_d_theta}. Moreover, it is not difficult to prove the following estimate:
$$
|u_{\theta} (1+ |u_{\theta}|)^{\theta } - v_{\theta} (1+ |v_{\theta}|)^{\theta }|
\leq (1+\theta)[1+(1-\theta)|\xi_\theta |]^{\frac{2\theta}{1-\theta}}
|U_\theta-V_\theta |,
$$
where $\xi_\theta$ is between $U_\theta=\Psi_{\theta } (u_\theta)$ and $V_\theta=\Psi_{\theta } (v_\theta)$. 
By Proposition \ref{first_prop}, we deduce that $(1+\theta)[1+(1-\theta)|\xi_\theta|]^{\frac{2\theta}{1-\theta}}\to 1$, as $\theta \to 0$,  uniformly in $B$.  

Now, set $W_{\theta } := \frac{U_{\theta}-V_{\theta}}{\Vert  U_{\theta}-V_{\theta}\Vert}_{L ^2(B)}$.
By the above identity, the norms $\Vert \nabla W_{\theta}\Vert_{L ^2(B)}$ are uniformly  bounded. Hence there is a function $\tilde{W}\in H^1 (B)$, such that along a subsequence,
$\nabla W_{\theta } \to \nabla \tilde{W} $ weakly in $L^2 (B)$ and  
$
W_{\theta} \to \tilde{W}$ in $L^2(B)$,
so that
$
\Vert \tilde{W} \Vert_{L ^2(B)} =1$. This also 
 implies  
$$ 
\Vert \nabla  \tilde{W} \Vert_{L ^2(B)}^2 \leq \lambda^2 (B) .
$$
Next we claim that $\displaystyle\left|\int_B W_\theta\, dx\right|\leq C\theta$. Indeed,
since $\displaystyle \int_B u_\theta\, dx= \int_B v_\theta\, dx=0$, we have
$$
\left| \int _B (U_\theta-V_\theta ) \, dx\right|
=
\int _B (\Psi _{\theta } (u_\theta) -u_\theta -\Psi _{\theta } (v_\theta)+v_\theta) \, dx
$$
$$
\leq  
\int _B \left| \int _{v_\theta} ^{u_\theta} (\Psi ' _{\theta } (t) -1) (u_\theta-v_\theta)dt \right|\, dx
\leq
\theta \int_B |u_\theta-v_\theta |\, dx\,.
$$
Now, $\Phi_{\theta }$ is locally Lipschitz continuous, uniformly in $\theta $. By Proposition \ref{first_prop}, we obtain
$$
\left|\int _B (U_{\theta }-V_{\theta } )\, dx \right|
\leq \theta \int_B |u_{\theta }-v_{\theta }|\, dx
\leq
C\theta \int_B |U_{\theta}-V_{\theta }| \, dx
\leq
C\theta \Vert U_{\theta }-V_{\theta }\Vert _{L ^2(B)}\,
$$
and the claim follows.
This and the fact that $W_{\theta} \to \tilde{W}$ in $L^2(B)$ prove that $\displaystyle \int_B \tilde{W}\, dx=0$.
Then, by definition of  $\lambda^2(B)$, we have that
$$ 
\Vert \nabla \tilde{W} \Vert_{L ^2(B)}^2 \geq  \lambda^2(B)\,.
$$
Hence $\tilde{W}$ is a (nonzero) eigenfunction for the Neumann Laplacian in $B$.
By the properties of $U_{\theta}$ and $V_{\theta}$ and by (\ref{u_theta_even}),
one has $W_{\theta}(x_1 , x_2 ) = W_{\theta}(-x_1 , x_2 ) =W_{\theta}(x_1 , -x_2 )$.
Thus the same symmetry properties hold for $\tilde{W}$. But this is in contradiction with the  shape of the eigenfunction in a ball, which is given by
$$
J_n(\alpha_{nk}|(x_1,x_2)|/R)\cdot \left\{
\begin{array}{ll}
\cos(n\varphi)\,, & l=1
\\
\sin(n\varphi)\,, & l=2 (n\neq 0)
\end{array}
\right.
,
$$
where we have used the polar coordinates, $R$ is the radius of the ball, and  $\alpha_{nk}$ are the positive roots of the derivative of the Bessel function $J_n$, (see for example
\cite{Henrot}). We thus have proved  that any minimizer is anti-symmetric.
\\
Note that the anti-symmetry also implies that $c_{\theta }=0 $, which can be seen by integrating (\ref{euler1}).

It remains to prove that the minimizer is unique for small $\theta $. Assume this is not the case. 
Therefore along a sequence  $\theta  \to 0$ along which there exist two distinct minimizers $ u_{\theta  } $ and $u_{\theta }'  $. Let the corresponding constants $d$ of (\ref{euler1}) be denoted by $d_{\theta } $ and $d_{\theta } ' $. Multiplying (\ref{euler1}) for $u_{\theta} $ with $u_{\theta } $, and integrating by parts gives 
\begin{eqnarray}
\label{dtheta} 
d_{\theta }  & = & 
- \int _B 
\frac{ |\nabla u_{\theta } |^2 }{(1+|u_{\theta } |) ^{2\theta } }\, dx+ 
\theta  
\int_B 
\frac{|\nabla u_{\theta } |^2 | u_{\theta } |}{(1+|u_{\theta }|)^{2\theta +1} } \, dx
\\
\nonumber
 & = &
 - \lambda ^{2,\theta }(B) + \theta    
 \int_B \frac{|\nabla u_{\theta } |^2 | u_{\theta } |}{(1+|u_{\theta }|)^{2\theta +1} } \, dx
.
\end{eqnarray}
Similarly
\begin{eqnarray}
\label{dthetaprime} 
d_{\theta }  ' & = &  - \lambda ^{2,\theta }(B) + \theta    
\int_B \frac{|\nabla u_{\theta } ' |^2 | u_{\theta } ' |}{(1+|u_{\theta }' |)^{2\theta +1} } \, dx
.
\end{eqnarray}
We define
$$
g_{\theta } (\xi) := \frac{|\xi |}{(1+|\xi |) ^{2\theta +1 }}  .
$$  
Since the functions $g_{\theta }$ are locally Lipschitz continuous, uniformly in $\theta $,  
we can estimate, using Proposition 1, 
\begin{eqnarray}
\label{estg}
 & & 
 \left| 
 \frac{|\nabla u_{\theta } |^2 | u_{\theta } |}{(1+|u_{\theta }|)^{2\theta +1} } 
-\frac{|\nabla u_{\theta } ' |^2 | u_{\theta } '|}{(1+|u_{\theta }' |)^{2\theta +1} } 
\right| 
\\
\nonumber 
 & = & 
 \left| 
 \frac{ |u_{\theta } |}{(1+|u_{\theta }|)^{2 \theta +1} } 
 \left( 
 |\nabla u_{\theta } |^2 - |\nabla u_{\theta } ' |^2 
 \right) 
 + |\nabla u_{\theta } ' | ^2 
 \left(  
 g_{\theta } (u_{\theta } ) -g_{\theta } (u_{\theta } ' ) 
 \right)
 \right| 
\\
\nonumber
 & \leq & C
 \left( 
 \left| 
 \nabla u_{\theta } -\nabla u_{\theta }' 
 \right| 
 +  
\left| 
u_{\theta } - u_{\theta }' 
\right| 
\right)
.
\end{eqnarray}
Subtracting (\ref{dthetaprime}) from (\ref{dtheta}) and taking into account (\ref{estg}), we obtain,
 \begin{eqnarray}
 \label{dminusdprime}
 |d_{\theta } - d_{\theta } ' |
  & \leq & 
  \theta  
  \int_B \left| \frac{|\nabla u_{\theta } |^2 | u_{\theta } |}{(1+|u_{\theta }|)^{2\theta +1} } 
-\frac{|\nabla u_{\theta } ' |^2 | u_{\theta } '|}{(1+|u_{\theta }' |)^{2\theta +1} } \right| \, dx
\\
\nonumber
 & \leq & C\theta \left( \Vert u_{\theta } -u_{\theta } ' \Vert_{L ^2(B)} + \Vert \nabla (u_{\theta } -u_{\theta } ') \Vert_{L ^2(B)} \right) .
 \end{eqnarray}
 Now we claim that 
  \begin{equation}\label{ddprime}
  |d_{\theta } - d_{\theta } ' |\le C\theta \Vert U_{\theta } -U_{\theta } ' \Vert_{L ^2(B)} \,.
 \end{equation}
As in the proof of Lemma \ref{lemma_c_theta}, by multiplying equation
 (4.3), (with $p=2$ and $g=0$), by $(1+|u_\theta|)^{\theta}$, we have
\begin{equation}\label{eqeulermultiplied2}
-\nabla ( (1+|u_{\theta} |)^{-\theta } \nabla u_{\theta }) + c_{\theta } (1+|u_\theta |)^{\theta } + d_{\theta } u_{\theta }(1+|u_{\theta } |)^{\theta }=0.
\end{equation}
 Moreover the analogous of this equality holds true for $u_{\theta}'$. Moreover by multiplying this equation by $u_\theta-u'_\theta$, we get
  \begin{eqnarray}
 \label{dd'prime1}
&&  \int_B \left( \frac{\nabla u_{\theta }  }{(1+|u_{\theta }|)^{\theta } } -
  \frac{\nabla u'_{\theta }  }{(1+|u'_{\theta }|)^{\theta } } 
 \right ) \nabla (u_\theta-u'_\theta)\, dx\\
& +& \nonumber
 d_ \theta  
 \int_B [u_{\theta}(1+|u_{\theta}|)^{\theta} - u'_{\theta}(1+|u'_{\theta}|)^{\theta} ] (u_\theta-u'_\theta)\, dx
\\
\nonumber
 & +& (d_{\theta } - d_{\theta } ' )
 \int_B [u'_{\theta}(1+|u_{\theta}|)^{\theta} (u_\theta-u'_\theta)\, dx=0\,.
 \end{eqnarray}
Now we evaluate the three integrals on the left-hand side. For value of $\theta$ small enough, we have
\begin{align}
 \label{ddprime2}
 \int_B& \left( \frac{\nabla u_{\theta }  }{(1+|u_{\theta }|)^{\theta } }-
  \frac{\nabla u'_{\theta }  }{(1+|u'_{\theta }|)^{\theta } } 
 \right ) \nabla (u_\theta-u'_\theta)\, dx\\
 \nonumber
 &=
  \int_B\frac{1}{(1+|u_\theta|)^\theta}|\nabla (u_\theta-u'_\theta)|^2\,dx+\\
 \nonumber
& +\int_B\nabla u'_\theta\cdot \nabla (u_\theta-u'_\theta)\left (  
 \frac{1}{(1+|u_\theta|)^\theta}-\frac{1}{(1+|u'_\theta|)^\theta}
   \right )\, dx\\
 \nonumber
&\ge \frac 12 \|\nabla (u_\theta-u'_\theta)\|_{{L ^2(B)}}^2-C \|\nabla (u_\theta-u'_\theta)\|_{{L ^2(B)}} \| u_\theta-u'_\theta\|_{{L ^2(B)}}\,.
  \end{align}
 Moreover, as  in the  calculation of $J$ in the previous arguments (see after \eqref{estimate_c_theta}), we get  
 \begin{equation}
 \label{dd'prime3}
\left|  d_ \theta  
 \int_B [u_{\theta}(1+|u_{\theta}|)^{\theta} - u'_{\theta}(1+|u'_{\theta}|)^{\theta} ] (u_\theta-u'_\theta)\, dx
\right|\le C\| u_\theta-u'_\theta\|_{{L ^2(B)}}^2
 \end{equation}
 \begin{equation}
 \label{ddprime4}
\int_B\nabla u'_\theta\cdot \nabla (u_\theta-u'_\theta)\left (  
 \frac{1}{(1+|u_\theta|)^\theta}-\frac{1}{(1+|u'_\theta|)^\theta}
   \right )\, dx\le C\theta \| u_\theta-u'_\theta\|_{{L ^2(B)}} (\| u_\theta-u'_\theta\|_{{L ^2(B)}}+\| \nabla (u_\theta-u'_\theta)\|_{{L ^2(B)}})\,.
 \end{equation}
Combining \eqref{ddprime2}- \eqref{ddprime4}, via Young inequality, we get
$$
\| \nabla (u_\theta-u'_\theta)\|_{{L ^2(B)}}^2\le C\|  u_\theta-u'_\theta\|^2_{{L ^2(B)}}\,.
$$
Now as in the previous calculation we get
$$
\|  u_\theta-u'_\theta\|_{{L ^2(B)}}\le \|  U_\theta-U'_\theta\|_{{L ^2(B)}}
$$
This gives \eqref{ddprime}.

 Next we define 
 $$
 h_{\theta } (\xi ) := \xi (1+ |\xi |)^{\theta } -\Psi _{\theta } (\xi ) =  \xi (1+ |\xi |)^{\theta } -\frac{\mbox{sgn } \xi }{ 1-\theta }
 \left[ (1+|\xi | )^{1-\theta } -1\right] .
 $$
 It is easy to see that $h_{\theta } $ is locally Lipschitz continuous with 
 $$
 | h' _{\theta } (\xi ) | \leq C \theta , \quad ( |\xi |\leq M  ),
 $$
 where the constant $C$ depends only on $M$, ($M>0$). 
 Using Proposition 1 we obtain from this 
 \begin{equation} 
 \label{estlast}
 \left| h_{\theta } (u_{\theta } )- h_{\theta } (u_{\theta } ' ) \right| 
 \leq 
 C\theta | u_{\theta } -u_{\theta } ' |.
 \end{equation}
 Now let
 $U_{\theta } := \Psi _{\theta } (u_{\theta } )$ and  $U_{\theta } ' := \Psi _{\theta } (u_{\theta } ' )$. 
 Arguing similarly as before, we first observe
 \begin{equation}
 \label{uversusU}
 |u_{\theta } -u_{\theta } | \leq C |\Psi _{\theta } (u_{\theta } ) - \Psi _{\theta } (u_{\theta } ' )| = 
 C |U_{\theta } - U_{\theta } ' |.
 \end{equation}
 We have by (\ref{euler1}),
\begin{eqnarray}
\label{eulerU}
 & & -\Delta U_{\theta } + d_{\theta } u_{\theta } (1+|u_{\theta } |)^{\theta } =0,
 \\
\label{eulerUprime}
 & & -\Delta U_{\theta } ' + d_{\theta } ' u_{\theta } ' (1+|u_{\theta } ' |)^{\theta } =0.
 \end{eqnarray}
 Subtracting (\ref{eulerUprime}) from (\ref{eulerU}), multiplying with $(U_{\theta } - U_{\theta } ' )$ and integrating by parts gives,
 \begin{eqnarray}
 \label{UminusUprime} 
  0 & = & 
\Vert \nabla (U_{\theta } -U_{\theta }' )\Vert _2 ^2 + d_{\theta } \int_B \left(    u_{\theta } (1+|u_{\theta } |)^{\theta } -u_{\theta }' (1+|u_{\theta } '|)^{\theta } \right) (U_{\theta } - U_{\theta } ' ) \, dx
\\
\nonumber 
 & & + (d_{\theta } -d_{\theta } ' ) \int_B u_{\theta } ' (1+|u_{\theta } ' |)^{\theta} (U_{\theta }-U_{\theta } ' ) \, dx.
\end{eqnarray}
Now define $W_{\theta } := (U_{\theta } - U_{\theta } ' ) \Vert U_{\theta } - U_{\theta } ' \Vert _{L ^2(B)} ^{-1} $. 
Then we obtain from (\ref{UminusUprime}),
 \begin{eqnarray}
 \label{Wtheta}
 0
   & = &
\Vert \nabla W_{\theta } \Vert _{L ^2(B)} ^2 + d_{\theta }\int_B \left(    u_{\theta } (1+|u_{\theta } |)^{\theta } -u_{\theta }' (1+|u_{\theta } '|)^{\theta } \right) W_{\theta } \Vert U_{\theta } - U_{\theta } ' \Vert _{L ^2(B)} ^{-1} \, dx
\\
\nonumber 
 & & 
+ (d_{\theta } -d_{\theta } ' ) \int_B u_{\theta } ' (1+|u_{\theta } ' |)^{\theta} W_{\theta } \Vert U_{\theta }-U_{\theta } ' \Vert _{L ^2(B)} ^{-1}  \, dx.
\\
\nonumber 
 & = & 
 \Vert \nabla W_{\theta } \Vert _{L ^2(B)} ^2 + d_{\theta } \Vert W_{\theta } \Vert _{L ^2(B)} ^2 
 \\
 \nonumber
  & & + d_{\theta }\int_B \left(   h_{\theta } (u_{\theta } )- h_{\theta } (u_{\theta } ')  \right) 
 W_{\theta } \Vert U_{\theta } - U_{\theta } ' \Vert _{L ^2(B)} ^{-1} \, dx
\\
\nonumber 
 & & 
+ (d_{\theta } -d_{\theta } ' ) \int_B u_{\theta } ' (1+|u_{\theta } ' |)^{\theta} W_{\theta } \Vert U_{\theta }-U_{\theta } ' \Vert _{L ^2(B)} ^{-1}  \, dx.
\end{eqnarray}
Since $\Vert W_{\theta }\Vert _{L ^2(B)} =1$,  and in view of the estimates 
(\ref{dminusdprime}) and (\ref{uversusU}), we see that the last two  terms in (\ref{Wtheta}) tend to zero as $\theta \to 0$.  Hence  the functions  $W_{\theta}$ are uniformly bounded in $H^1 (B)$. 
By passing to another subsequence if necessary, we find a function $\overline{W} \in H^1 (B)$  such that 
$W_{\theta } \to \overline{W} $ weakly in $H^1 (B)$ and $W_{\theta } \to \overline{W} $ in $L^2 (B)$. 
Then, passing to the limit in (\ref{Wtheta}) we obtain, since $\liminf\limits _{\theta \to 0} \Vert \nabla W_{\theta } \Vert _{L ^2(B)} \geq \Vert \nabla \overline{W} \Vert _{L ^2(B)} $,
\begin{equation}
\label{Wlimit1}
\Vert \nabla \overline{W} \Vert _{L ^2(B)} ^2 \leq \lambda ^2  (B) \Vert \overline{W} \Vert _{L ^2(B)} ^2 .
\end{equation}
Since also $\displaystyle \int_B \overline{W} \, dx=0$ and $\displaystyle \int_B \overline{W} ^2 \, dx=1 $, we must have equality in (\ref{Wlimit1}), and $\overline{W} $ is an anti-symmetric eigenfunction for the Neumann Laplacian in $B$, that is, $\overline{W} = u$.
In other words, we have 
$$
\int_B W_{\theta } U_{\theta }\, dx \to \int_B u^2\, dx =1, \quad \mbox{as $\theta \to 0$.}
$$
On the other hand, we calculate
\begin{eqnarray*}
\int_B W_{\theta } U_{\theta } \, dx
 & = &  \displaystyle
 \frac{\int_B U_{\theta } ^2\, dx - \int_B U_{\theta } ' U_{\theta } \, dx}{
\Vert U_{\theta } -U_{\theta } ' \Vert _{L ^2(B)}} 
\\ 
  & = &  \displaystyle \frac{1 - \int_B U_{\theta } ' U_{\theta } \, dx}{
\sqrt{2  -2 \int_B U_{\theta } ' U_{\theta }\, dx }} 
\\ 
 & = & 
\frac{1}{\sqrt{2}}{ \sqrt{ 1 -\int_B U_{\theta } ' U_{\theta } \, dx}}
\\ 
& & \to 0 , \quad \mbox{ as } \  \theta \to 0,
\end{eqnarray*}
which gives a contradiction. The proof is complete. 
\end{proof}


\section{Symmetry breaking  in dimension 2}
In this section we continue  studying the two dimensional case, assuming again that $F\equiv 0$.
We show that   for $p$ sufficiently large
the 
minimizers of $\lambda^{\theta,p}$  do not verify the properties of anti-symmetry described in the previous section;  
therefore  a phenomenon of symmetry breaking  occurs.


Let us denote by $W^{1,q}_{as}(B)$ the subset of the Sobolev space $W^{1,q}(B)$ of the functions which are anti-symmetric with respect to the plane $P\equiv \{x\in \R^{N+1}\, : x_N=0 \}$, that is,
$$
W^{1,q}_{as}(B):=\left\{ v\in  W^{1,q}(B)\, :\, u(x', -x_N)=-  u(x', x_N)\right\}\,.
$$
Let 
$$
\mathcal{F}(v)=\int_{B}\frac{|\nabla v|^2}{(1+|v|)^{2\theta}}\, dx,\,\,\,\,\,\, v\in W^{1,q}(B), \, v\neq 0 \\, \int_{B}v\, dx=0, \, \norm{v}_{L^2(B)}=1\,.
$$
Recall that  $\theta$ satisfies \eqref{teta} and $q$ satisfies \eqref{q2}.
Let
$$
\lambda^{\theta,p}(B):=\hbox{inf} 
\left\{
\mathcal{F}(v) ,\,\, v\in W^{1,q}(B), \, v\neq 0 \\, \int_{B}v\, dx=0, \, \norm{v}_{L^p(B)}=1 
\right\}
$$
and 
$$
{\lambda}^{\theta,p}_{as}(B):=
\hbox{inf} 
\left\{
\mathcal{F}(v),\,\, v\in W^{1,q}_{as}(B), \, v\neq 0 \\, \int_{B}v\, dx=0, \, \norm{v}_{L^p(B)}=1 
\right\}\, .
$$
Observe that the existence of a function realizing ${\lambda}^{\theta,p}_{as}(B)$
can be proved analogously as in Theorem \ref{existence}.
%
%
%
%

\noindent Let us also recall  a well-known result. For any bounded smooth domain $\Omega\subset \R^2 $, 
let 
$$
{\Lambda}^{p}_{as}(\Omega)=\hbox{inf} 
\left\{
\norm{\nabla v}^2_{L^2(\Omega)},\,\, v\in W^{1,2}_{as}(\Omega), \, v\neq 0 \\, \int_{\Omega}v\, dx=0, \, \norm{v}_{L^p(\Omega)}=1 
\right\}\,.
$$
In \cite{Girao-Weth}  the  behaviour of ${\Lambda}^{p}_{as}(\Omega)$ is studied and it is proved that 
\begin{equation}\label{lemma_girao-weth-ell}
{\Lambda}^{p}_{as}(\Omega)\to 0, \qquad \hbox{as } p\to \infty.
\end{equation}
It is easy to prove the same result for ${\lambda}^{\theta,p}_{as}(B)$:
\begin{proposition}\label{lemmaantisimm}
We have ${\lambda}^{\theta,p}_{as}(B)\to 0$, as $p\to \infty$.
\end{proposition}

\begin{proof}
Since $\displaystyle \int_B |\nabla v|^2\, dx\geq \mathcal{F}(v)$, one has
$$
\begin{array}{ccl}
\displaystyle
{\Lambda}^{p}_{as}(B) & \geq & 
\displaystyle
\hbox{inf} \left\{
\mathcal{F}(v),\,\, v\in W^{1,2}_{as}(B), \, v\neq 0 \, \int_{B}v\, dx=0, \, \norm{v}_{L^p(B)}=1 
\right\}
\\
\displaystyle&\geq & \displaystyle
\hbox{inf} \left\{
\mathcal{F}(v),\,\, v\in W^{1,q}_{as}(B), \, v\neq 0 \, \int_{B}v\, dx=0, \, \norm{v}_{L^p(B)}=1 
\right\}
\\
\displaystyle
& =& {\lambda}^{\theta,p}_{as}(B)
\end{array}
$$
By \eqref{lemma_girao-weth-ell} the conclusion follows.
\end{proof}

Now we can prove the main result of the section.
\begin{theorem}
For  $p$ sufficiently large, $\lambda^{\theta,p}(B)<{\lambda}^{\theta,p}_{as}(B)$. Therefore the minimizers of $\mathcal{F}$ are not anti-symmetric for $p$ sufficiently large.
\end{theorem}\label{antis2}
\begin{proof}
Let $v_p$ be an eigenfunction for ${\lambda}^{\theta,p}_{as}(B)$. Hence 
$\norm{v_p}_{L^p(B)}=1$.
Let $B_+=\{(x_1,x_2)\in B: x_2>0 \}$, and 
let $\overline{u}_p$ be defined by 
$$
\overline{u}_p(x)=
\left\{
\begin{array}{ll}
v_p(x),& x\in B_+,
\\
0, & x\in B\setminus B_+\,.
\end{array}
\right.
$$
Then 
\begin{equation}\label{lastestimate}
\int_{B}\frac{|\nabla \overline{u}_p|^2}{(1+|\overline{u}_p|)^{2\theta}}\, dx=\frac{{\lambda}^{\theta,p}_{as}(B)}{2}\,,\qquad \qquad \norm{\overline{u}_p}_{L^p(B)}^p=\frac {1}{2}. 
\end{equation}
We claim that 
\begin{equation}\label{mean}
 \int_B \overline{u}_p\, dx\to 0, \qquad \hbox{as $p \rightarrow \infty$}. 
\end{equation}

\noindent By Proposition \ref{lemmaantisimm}, we deduce that
$$
{\lambda}^{\theta,p}_{as}(\Omega)=2\|\nabla \Psi(\overline{u}_p)\|_{L^2(B)}\to 0, \qquad \hbox{as }p\to \infty\,,
$$
where $\Psi$ has been defined in (\ref{defnPsi}).
Since  $\overline{u}_p=0$ in $B\setminus B_+$, we can use  
Poincar\'e-Wirtinger inequality which implies  
$$
\|\Psi(\overline{u}_p)\|_{L^2(B)}\to 0,\qquad \hbox{as }p\to \infty\,.
$$ 
Therefore, up to subsequence,  $\Psi(\overline{u}_{p})\to 0$ and $u_p\to 0$ a.e. in $B$.
On the other hand, there exists a function $h \in L^2(B)$ such that $|\Psi(\overline{u}_p)|\leq h$ a.e. in $B$. By definition of $\Psi(t)$, we deduce the existence of a function $k\in L^{2(1-\theta)}(B)$ such that
$|\overline{u}_p|\leq k$ a.e. in $B$. Hence Lebesgue's theorem applies and we get  $\displaystyle \int_B |\overline{u}_p|\, dx\to 0$. This proves \eqref{mean}.

Next we define 
$$
\tilde{u}_p:=\frac{\overline{u}_p-\frac{1}{|B|}\displaystyle\int_B \overline{u}_p\, dx}{ \left\Vert \overline{u}_p-\displaystyle\frac{1}{|B|}\int_B \overline{u}_p\, dx \right\Vert _{L^p(B)}}\,.
$$
Therefore
$$
\lambda^{\theta,p}(B)
\leq \int_{B}\frac{|\nabla \tilde{u}_p|^2}{(1+|\tilde{u}_p|)^{2\theta}}\, dx=
\frac{1}{\Big\|\overline{u}_p-\frac{1}{|B|}\displaystyle\int_B \overline{u}_p\, dx \Big\|_{L^p(B)}^{2}}
\displaystyle\int_B\frac{|\nabla \overline{u}_p|^2}{\left(1+\frac{|\overline{u}_p - \frac{1}{|B|} \int_B \overline{u}_p\, dx|}{\|\overline{u}_p - \frac{1}{|B|}\int_B \overline{u}_p \, dx \|_{L^p(B)}}\right)^{2\theta}}\, dx\,.
$$
Let $\varepsilon>0$ be sufficiently small.
For a suitable $p(\varepsilon)>0$ and for any $p>p(\varepsilon)$, one has, by (\ref{lastestimate})
\begin{equation}\label{estimatebeforeG}
\lambda^{\theta,p}(B)
\leq\frac{1}{\left[(\frac 12)^{\frac 1p}-\varepsilon\right]^{2}}
\int_B\frac{|\nabla \overline{u}_p|^2}{\left[1+\frac{|\overline{u}_p - \frac{1}{|B|}\int_B \overline{u}_p\, dx|}{(\frac 12)^{\frac 1p}+\varepsilon}\right]^{2\theta}}dx.
\end{equation}
Let us set $\displaystyle M_\varepsilon=\frac{1+\varepsilon}{1-\varepsilon}$.

\noindent We claim that
\begin{equation}\label{G}
\displaystyle G(\overline{u}_p)\equiv\frac{1+|\overline{u}_p|}{1+\frac{|\overline{u}_p-\frac{1}{|B|}\int_B \overline{u}_pdx|}{(1/2)^{\frac 1p}+\varepsilon}}\leq M_\varepsilon\,.
\end{equation}
First of all, it is easy to verifies that, for any $p>p(\varepsilon)$,
\begin{equation}\label{first}
1+\left|\overline{u}_p-\frac{1}{|B|}\int_B \overline{u}_pdx\right|
\geq
1+\left||\overline{u}_p|-\frac{1}{|B|}\left|\int_B \overline{u}_pdx\right|\right|
\geq 
1+||\overline{u}_p|-\varepsilon|
\geq
(1-\varepsilon)(1+|\overline{u}_p|)\,.
\end{equation}
Now we distinguish two cases.
\begin{enumerate}
\item
If $(\frac 12)^{\frac 1p}+\varepsilon\leq 1$, then 
$G(\overline{u}_p)\leq \frac{1+|\overline{u}_p|}{1+{|\overline{u}_p-\frac{1}{|B|}\int_B \overline{u}_pdx|}}$. By (\ref{first}) one has
$G(\overline{u}_p)\leq \frac{1}{1-\varepsilon}$.
\item
If $(\frac 12)^{\frac 1p}+\varepsilon>1$, then  by (\ref{first}),
$$1+\frac{|\overline{u}_p-\frac{1}{|B|}\int_B \overline{u}_pdx|}{(\frac 12)^{\frac 1p}+\varepsilon}
\geq 
\frac{1+|\overline{u}_p-\frac{1}{|B|}\int_B \overline{u}_pdx |}{(\frac 12)^{\frac 1p}+\varepsilon}
\geq 
\frac{1-\varepsilon}{(\frac 12)^{\frac 1p}+\varepsilon}(1+|\overline{u}_p|)\,.
$$
\end{enumerate}
Therefore \eqref{G} is proved, that is, 
\begin{equation}\label{estimateG}
\frac{1}{1+\frac{|\overline{u}_p-\frac{1}{|B|}\int_B \overline{u}_pdx|}{(1/2)^{\frac 1p}+\varepsilon}}
\leq \frac{M_\varepsilon}{1+|\overline{u}_p|}\,.
\end{equation}
Combining estimates  (\ref{estimatebeforeG}) and (\ref{estimateG}) we get
$$
\lambda^{\theta,p}(B)
\leq
\frac{M_\varepsilon^{2\theta}}{\left[(\frac 12)^{\frac 1p}-\varepsilon\right]^{2}}
\int_B
\frac{|\nabla \overline{u}_p|^2}{(1+|\overline{u}_p|)^{2\theta}}dx\,.
$$
It is clear that
$\displaystyle\frac{M_\varepsilon^{2\theta}}{\left[(\frac 12)^{\frac 1p}-\varepsilon\right]^{2}}
=\left(\frac{1+\varepsilon}{1-\varepsilon}\right)^{2\theta}\frac{1}{\left[(\frac 12)^{\frac 1p}-\varepsilon\right]^{2}}<2$
for $p>p(\varepsilon)$.
Therefore, for $p$ sufficiently large, one has
$$
\lambda^{\theta,p}(B)
<
2
\int_B
\frac{|\nabla \overline{u}_p|^2}{(1+|\overline{u}_p|)^{2\theta}}dx
={\lambda}^{\theta,p}_{as}(B)
$$
by (\ref{lastestimate}).
\end{proof}


\noindent{\bf Acknowledgements.} 
This work started during a visit of A. Mercaldo to Universit\'e de Rouen which was 
financed by the F\'ed\'eration Normandie Math\'ematiques. The last author
 is a member of Gruppo Nazionale per
l'Analisi Matematica, la Probabilit\`a e le loro Applicazioni (GNAMPA)
of the Istituto Nazionale di Alta Matematica (INdAM)
which supported visitings
of F. Brock  to Universit\`a degli Studi di Napoli Federico II.
All these institutions are gratefully acknowledged.

\end{document}